\documentclass[11pt]{article}

\usepackage[margin=1in]{geometry}

\usepackage{graphicx}
\graphicspath{{figures/}}
\usepackage{mathtools}

\usepackage{amsmath,amsthm,amsfonts}
\usepackage{latexsym}
\usepackage{amssymb}

\usepackage{bbm}

\usepackage{float}

\usepackage{enumitem}

\usepackage{color}

\usepackage{hyperref}

\usepackage{subcaption}

\usepackage{cite}

\usepackage{xspace}

\usepackage{algorithm}
\usepackage{algpseudocode}

\newtheorem{theorem}{Theorem}[section]

\newtheorem{lemma}[theorem]{Lemma}

\theoremstyle{remark}
\newtheorem{remark}[theorem]{Remark}

\newcommand{\R}{\mathbb{R}}

\newcommand{\Ga}{\Gamma}

\DeclareMathOperator{\Mod}{Mod}

\DeclareMathOperator{\supp}{supp}

\author{Nathan Albin, Kapila Kottegoda and Pietro Poggi-Corradini\\
\footnotesize{Department of Mathematics, Kansas State University, Manhattan, KS}}
\title{An exact-arithmetic algorithm for spanning tree modulus}

\begin{document}
\maketitle
\begin{abstract}
Spanning tree modulus is a generalization of effective resistance that is closely related to graph strength and fractional arboricity. The optimal edge density associated with spanning tree modulus is known to produce two hierarchical decompositions of arbitrary graphs, one based on strength and the other on arboricity. Here we introduce an exact-arithmetic algorithm for spanning tree modulus and the strength-based decomposition using Cunningham's algorithm for graph vulnerability.  The algorithm exploits an interesting connection between spanning tree modulus and critical edge sets from the vulnerability problem.  This paper introduces the new algorithm, describes a practical means for implementing it using integer arithmetic, and presents some examples and computational time scaling tests.
\end{abstract}

\section{Introduction}\label{sec:introduction}

Let $G=(V,E)$ be a finite, connected, undirected graph. 
We allow parallel edges, but self-loops are irrelevant and will be removed. There are two quantities that can be computed for $G$ and that describe the relationship between the edge set $E(G)$ and the family $\Ga$ of all spanning trees of $G$. These are the strength and the fractional arboricity. The {\it strength} is the largest constant $S(G)$ that allows one to control the number of additional connected components that are created when removing $A$ from $G$, using the size of $A$. In other words, $S(G)$ is the largest constant with the property that
\[
|A|\ge S(G) (k_A-1),\qquad\forall A\subset E(G),
\]
where $k_A$ is the number of connected components of $G-A.$

On the other hand, the {\it fractional arboricity} is the smallest constant $D(G)$ such that
\[
|A|\le D(G) (|V(A)|-1),\qquad\forall A\subset E(G),
\]
Note that we always have the following inequalities:
\[
S(G)(|V(G)|-1)\le |E(G)| \le D(G)(|V(G)|-1).
\]
As was shown in \cite{achpcst:disc-math2021} (see also \cite{kapila2020,huy-pc:fulkerson,huy-pc:matroid-mod,huy-pc:reinforce}) spanning tree modulus is a convex optimization problem on $G$ that admits a unique optimal density $\eta^*\in \R_{\ge 0}^E$. Moreover, the minimum, $\eta^*_{\rm min}$, is equal to $D(G)^{-1},$ and the maximum, $\eta^*_{\rm max}$, is equal to $S(G)^{-1}.$ In fact, every intermediate value of $\eta^*$ has an interpretation in terms of the graph $G$ and some natural decompositions.

Consequently, an algorithm capable of computing the edge densities $\eta^*$ allows the study of interesting combinatorial properties of graph structures. In this paper, we develop an approach, based on Cunningham's algorithm~\cite{cunningham1985optimal}, that computes the values of $\eta^*$ in exact arithmetic.

\subsection{Spanning tree modulus}
We begin with a brief review of spanning tree modulus (see, e.g.,~\cite{achpcst:disc-math2021,kapila2020,kottegoda2020}).  Consider the family, $\Gamma$, of all spanning trees of a connected, undirected graph $G=(V,E)$.  ($G$ need not be simple.)  A non-negative function $\rho:E\to\mathbb{R}_{\ge 0}$ is called a \emph{density} on $G$.  Each density, $\rho$, gives every spanning tree $\gamma\in\Gamma$ a \emph{$\rho$-length} (or \emph{$\rho$-weight}) $\ell_\rho(\gamma)$ defined as
\begin{equation*}
    \ell_\rho(\gamma) := \sum_{e\in \gamma}\rho(e).
\end{equation*}

The \emph{spanning tree modulus} of $G$ is defined as the solution to the optimization problem
\begin{equation}\label{eq:mod-def}
\begin{split}
\text{minimize}\quad&\sum_{e\in E}\rho(e)^2\\
\text{subject to}\quad &\ell_\rho(\gamma)\ge 1\quad\text{for all $\gamma\in\Gamma$},\\
& \rho(e)\ge 0\quad\text{for all $e\in E$}.
\end{split}
\end{equation}
The decision variables are the values of the \emph{density} $\rho$ assigned to each edge $e\in E$.  The density should be non-negative and is required to give at least unit $\rho$-length to every spanning tree.  The minimum value in~\eqref{eq:mod-def} is called the spanning tree modulus of $G$, denoted by $\Mod(\Gamma)$. The problem admits a unique minimizing density, denoted by $\rho^*$, which is called the \emph{optimal density}.  Thus,
\begin{equation*}
    \Mod(\Gamma) = \sum_{e\in E}\rho^*(e)^2.
\end{equation*}

At first glance, spanning tree modulus appears to be computationally challenging.  The key difficulty lies in the inequality constraints, of which there are $|\Gamma|+|E|$.  Since $|\Gamma|$ tends to be combinatorially large in $|E|$, it is not immediately apparent that an efficient algorithm should exist; indeed, we cannot even hope to identify all constraints in polynomial time.

Nevertheless, there have been some indications that an efficient algorithm does exist.  For example, the iterative approximation algorithm introduced in~\cite{Albin2017} (adapted to the spanning tree setting) performs very well in practice, producing an accurate approximation to modulus even on graphs that have enormous numbers of spanning trees.  Similarly, the Plus-1 algorithm found in~\cite{clemens2018spanning} is able to produce rational approximations to spanning tree modulus with a known convergence rate.  In this paper, we take a different approach, building on Cunningham's graph vulnerability algorithm to produce an exact-arithmetic algorithm for spanning tree modulus.

\subsection{Fairest Edge Usage}

Spanning tree modulus has an interesting dual interpretation in the form of the \emph{Fairest Edge Usage (FEU)} problem~\cite{achpcst:disc-math2021}.  The FEU problem is understood through the context of random spanning trees.  In this interpretation, we consider a spanning tree $\underline{\gamma}$ chosen at random according to a probability mass function (pmf) $\mu$.  (The underline in the notation $\underline{\gamma}$ is used to distinguish the random tree from its possible values.)  In other words, for each $\gamma\in\Gamma$, $\mu(\gamma)$ defines the probability that $\underline{\gamma}=\gamma$ or, in simpler notation, $\mu(\gamma)=\mathbb{P}_\mu(\underline{\gamma}=\gamma)$.  Here we use the subscript notation $\mathbb{P}_\mu$ to specify exactly which pmf is being used.  We will also represent the relationship between the random spanning tree $\underline{\gamma}$ and its pmf $\mu$ by the notation $\underline{\gamma}\sim\mu$.  The set of all pmfs on $\Gamma$ is denoted by $\mathcal{P}(\Gamma)$.

If $e\in E$, $\mu\in\mathcal{P}(\Gamma)$ and $\underline{\gamma}\sim\mu$, then there is a natural concept of \emph{edge usage probability} for each edge.  Namely, what is the probability that $e\in\underline{\gamma}$?  We indicate this quantity with the notation
\begin{equation*}
\eta(e) := \mathbb{P}_\mu(e\in\underline{\gamma}) =
\sum_{\gamma\in\Gamma}\mathbbm{1}_{e\in\gamma}\mu(\gamma),
\quad\text{where}\quad
\mathbbm{1}_{e\in\gamma} :=
\begin{cases}
1 & \text{if }e\in\gamma\\
0 & \text{if }e\notin\gamma.
\end{cases}
\end{equation*}
The FEU problem is the solution to another optimization problem, namely
\begin{equation}\label{eq:FEU}
\begin{split}
\text{minimize}\quad&\sum_{e\in E}\eta(e)^2\\
\text{subject to}\quad&\mu\in\mathcal{P}(\Gamma)\\
& \eta(e) = \mathbb{P}_\mu(e\in\underline{\gamma})\quad\text{for all }e\in E.
\end{split}
\end{equation}
Here, the decision variable is the pmf $\mu$ and the objective is to minimize the 2-norm of the associated edge usage probabilities given by $\eta$.

The optimization problem given in~\eqref{eq:FEU} is called the Fairest Edge Usage problem because it is equivalent to the problem of minimizing the variance of $\eta$.  Roughly speaking, the goal is to assign a pmf to the spanning trees of $G$ in such a way that the edge usage probabilities are as evenly distributed as possible.

The important connections between~\eqref{eq:mod-def} and~\eqref{eq:FEU} in the context of the present work are summarized in the following theorem, collected from those presented in~\cite{achpcst:disc-math2021,albin2019blocking,albin2016minimal}.

\begin{theorem}\label{thm:mod-summary}
For a given, nontrivial, connected, undirected graph $G=(V,E)$, the following are true.
\begin{enumerate}
    \item Problem~\eqref{eq:mod-def} admits a unique minimizing density $\rho^*$.
    \item Problem~\eqref{eq:FEU} admits at least one optimal pmf $\mu^*$.  While this pmf may not be unique, all optimal pmfs induce the same optimal edge usage probabilities, $\eta^*$.  In fact, a pmf is optimal for~\eqref{eq:FEU} if and only if it induces this optimal $\eta^*$.
    \item The optimal $\rho^*$ for~\eqref{eq:mod-def} and the optimal $\eta^*$ for~\eqref{eq:FEU} satisfy the following properties.
    \begin{enumerate}
        \item $\left(\sum\limits_{e\in E}\rho^*(e)^2\right)
        \left(\sum\limits_{e\in E}\eta^*(e)^2\right) = 1$, and
        \item $\eta^*(e) = \frac{\rho^*(e)}{\Mod(\Gamma)}\quad$
        for all $e\in E$.
    \end{enumerate}
\end{enumerate}
\end{theorem}

\subsection{Graph vulnerability}

Finally, we review the concept of graph vulnerability introduced in~\cite{cunningham1985optimal}.  We shall adopt a notation similar but not identical to the notations of~\cite{cunningham1985optimal} and~\cite{gueye2010design}.

For any subset of edges $J\subseteq E$, define
\begin{equation}\label{eq:theta-J}
\mathcal{M}(J) := \min_{\gamma \in \Gamma } |\gamma \cap J|\quad \text{and}\quad \theta(J) := 
\begin{cases}
\frac{\mathcal{M}(J)}{|J|} & \text{if }J \ne \emptyset,\\
\phantom{00}0 & \text{if }J = \emptyset.
\end{cases}
\end{equation}
The quantity $\mathcal{M}(J)$ is the minimum possible overlap between the edge set $J$ and any spanning tree of $G$. The value $\theta(J)$ is called the \emph{vulnerability} of $J$.  The \emph{vulnerability} of the graph $G$, $\theta(G)$, is the maximum vulnerability of subsets of its edges.  That is
\begin{equation}\label{eq:optimization problem}
\theta(G) := \max_{J\subseteq E}\theta(J).
\end{equation}
A set $J$ is said to be \emph{critical} if
\begin{equation}\label{eq:critical-set}
\theta(J) = \theta(G).
\end{equation}
Since $J=E$ is one possible choice, there is a simple lower bound for any nontrivial graph:
\begin{equation*}
\theta(G) \ge \frac{|V|-1}{|E|} > 0.
\end{equation*}
Thus, the empty set is never critical for a nontrivial graph.  There can be more than one critical set.  A key result of~\cite{cunningham1985optimal} is the existence of a polynomial-time algorithm for computing $\theta(G)$.  A method for additionally finding a critical set is presented in Section~\ref{sec:obtain-crit-set}.

\subsection{Contributions of this work}

The primary contribution of this work is the construction of an efficient algorithm for computing spanning tree modulus and the associated edge density, $\eta^*$, in exact arithmetic.  The remainder of this paper is organized as follows.  In Section~\ref{sec:vulnerability-to-modulus}, we show that an algorithm capable of finding a critical set can be used recursively to find the optimal edge usage probabilities $\eta^*$.  In Section~\ref{sec:graph-vulnerability}, we review Cunningham's polynomial-time algorithm for graph vulnerability.  In Section~\ref{sec:cunningham-mods}, we introduce some modifications to Cunningham's algorithm.  One of these is necessary in order to compute spanning tree modulus, the other is a practical modification that allows the use of integer arithmetic in implementations of the spanning tree modulus algorithm.  Finally, in Section~\ref{sec:spt-mod-algorithm}, we detail the final algorithm for spanning tree modulus, and present some computational examples.

\section{Using vulnerability to compute modulus}\label{sec:vulnerability-to-modulus}

In this section, we develop in more detail the connection between graph vulnerability and spanning tree modulus.  Most of the following results are consequences of the theorems presented in~\cite{achpcst:disc-math2021}, but are not stated exactly as we need them here.  In the interest of keeping the present work self-contained, we present direct proofs of the results we need here.  In what follows, $\mu^*$ represents any optimal pmf for spanning tree modulus and $\eta^*$ represents the corresponding edge usage probabilities.  If $J\subseteq E$ is a subset of edges, we denote by $G_J=(V,J)$ the subgraph of $G$ induced by the edges in $J$ and we let $Q(G_J)$ be the number of connected components of $G_J$.

The following two subgraph decompositions play an important role in the following discussion.  Suppose $J\subseteq E$.  When the edges of $J$ are removed from $G$, we are left with the graph $G_{\overline{J}}$, where $\overline{J}=E\setminus J$ is the complementary set of edges to $J$.  This graph has a number $k=Q(G_{\overline{J}})$ of connected components, call them $G_i'=(V_i,E_i')$ for $i=1,2,\ldots,k$.  The sets $\{V_1,V_2,\ldots,V_k\}$ form a partition of the vertex set $V$ and the sets $\{J,E_1',E_2',\ldots,E_k'\}$ form a partition of the edge set $E$.  The set $J$ also induces a related family of subgraphs $G_i=(V_i,E_i)$.  Each $G_i$ has the same vertex set as $G_i'$ and has edge set
\begin{equation*}
E_i = \{e\in E : e\text{ is incident upon two vertices in }V_i\}.
\end{equation*}
In other words, each $G_i$ is the vertex-induced subgraph that is induced by $V_i$.  Notice that $E_i'\subseteq E_i$ and that $E_i\setminus E_i'\subseteq J$.  The following lemma collects some well-known facts about spanning trees.

\begin{lemma}\label{lem:M-Q}
Let $J\subseteq E$.  Then
\begin{equation*}
\mathcal{M}(J) = Q(G_{\overline{J}})-1.
\end{equation*}
Moreover, if $\{G_i'\}$ is the associated collection of subgraphs induced by $J$ as described above and if $\gamma$ is any spanning tree of $G$, then $|\gamma\cap J| = \mathcal{M}(J)$ if and only if $\gamma\cap E_i'$ is a spanning tree of $G_i'$ for each $i$.
\end{lemma}
\begin{proof}
One way to see this is by a counting argument.  For any spanning tree $\gamma$ and any subgraph $G_i'=(V_i,E_i')$, $\gamma\cap E_i'$ is acyclic and, therefore, $|\gamma\cap E_i'|\le |V_i|-1$, with equality holding if and only if $\gamma\cap E_i'$ is a spanning tree of $G_i'$.

This implies that
\begin{equation*}
|\gamma\cap J| = |\gamma| - \sum_{i=1}^k|\gamma\cap E_i'|
\ge (|V|-1) - (|V|-k) = Q(G_{\overline{J}})-1.
\end{equation*}
This lower bound can only be attained if $|\gamma\cap E_i'|=|V_i|-1$ for each $i$.  To see that this is possible, let $\gamma_i$ be a spanning tree for $G_i'$ with $i=1,2,\ldots,k$.  Since $\cup_i\gamma_i$ is an acyclic subset of $E$, there must be a spanning tree, $\gamma$ of $G$ that contains this union.  But then 
\begin{equation*}
    |V|-1 = |\gamma| = |\gamma\cap J| + \sum_{i=1}^k|\gamma\cap E_i'| = |\gamma\cap J| + |V| - k,
\end{equation*}
implying that
\begin{equation*}
    \mathcal{M}(J) = |\gamma\cap J| = k-1.
\end{equation*}
\end{proof}

The following lemma establishes a necessary condition for optimality in~\eqref{eq:FEU}.

\begin{lemma}\label{lem:supp-mu-min-eta}
If $\gamma\in\supp\mu^*$, then
\begin{equation*}
    \ell_{\eta^*}(\gamma) = \min_{\gamma'\in\Gamma}\ell_{\eta^*}(\gamma').
\end{equation*}
That is, an optimal $\mu^*$ is supported only on trees with minimum $\eta^*$-length.
\end{lemma}
\begin{proof}
Suppose, to the contrary, that there exists a spanning tree $\gamma'\in\supp\mu^*$ that does not have minimum $\eta^*$-length and let $\gamma''$ be any minimum $\eta^*$-length spanning tree.  For $\epsilon\in[0,1]$, consider the pmf
\begin{equation*}
\mu_\epsilon = (1-\epsilon)\mu^* + \epsilon\delta_{\gamma''}.
\end{equation*}
The corresponding edge usage probability for this pmf is
\begin{equation*}
\eta_\epsilon(e) = (1-\epsilon)\eta^*(e) + \epsilon\mathbbm{1}_{e\in\gamma''}.
\end{equation*}
Define
\begin{equation*}
\begin{split}
g(\epsilon) &:= \sum_{e\in E}\eta_\epsilon(e)^2 =
\sum_{e\in E}\left(
(1-\epsilon)\eta^*(e) + \epsilon\mathbbm{1}_{e\in\gamma''}
\right)^2\\
&= (1-\epsilon)^2g(0) + 
2\epsilon(1-\epsilon)\sum_{e\in\gamma''}\eta^*(e)
+ \epsilon^2(|V|-1) \\
&= g(0) +
2\epsilon\left(
\ell_{\eta^*}(\gamma'') - g(0)
\right)
+ O(\epsilon^2).
\end{split}
\end{equation*}

But
\begin{equation*}
\begin{split}
g(0) &= \sum_{e\in E}\eta^*(e)^2 =
\sum_{e\in E}\eta^*(e)\sum_{\gamma\in\Gamma}\mathbbm{1}_{e\in\gamma}\mu^*(\gamma)\\
&= \sum_{\gamma\in\Gamma}\mu^*(\gamma)
\sum_{e\in E}\eta^*(e)\mathbbm{1}_{e\in\gamma}
= \sum_{\gamma\in\Gamma}\mu^*(\gamma)\ell_{\eta^*}(\gamma)
> \ell_{\eta^*}(\gamma''),
\end{split}
\end{equation*}
where the strict inequality arises from the facts that $\gamma'\in\supp\mu^*$ and $\ell_{\eta^*}(\gamma')>\ell_{\eta^*}(\gamma'')$.  Thus, for sufficiently small $\epsilon>0$, 
\begin{equation*}
\sum_{e\in E}\eta_\epsilon(e)^2 = g(\epsilon) < g(0)
= \sum_{e\in E}\eta^*(e)^2,
\end{equation*}
which contradicts the optimality of $\eta^*$.
\end{proof}

The set of edges where $\eta^*$ attains its maximum plays an important role in our results.  To this end, define
\begin{equation*}
    E^* := \{e\in E:\eta^*(e)=\eta^*_{\max}\},\quad\text{where}
    \quad\eta^*_{\max} := \max_{e\in E}\eta^*(e).
\end{equation*}

\begin{lemma}\label{lem:Estar-min-isect}
If $\gamma\in\supp\mu^*$, then
\begin{equation*}
|\gamma\cap E^*| = \mathcal{M}(E^*).
\end{equation*}
\end{lemma}
\begin{proof}
To see why this must be true, let $\gamma\in\supp\mu^*$ and consider constructing another spanning tree $\gamma'\in\Gamma$ as follows.  First, remove all edges of $E^*$ from $\gamma$, leaving a forest $\gamma\setminus E^*$.  Now, proceed as one would in Kruskal's algorithm, successively adding back an edge with the smallest $\eta^*$ possible without creating a cycle.  After some number of edge additions, a spanning tree $\gamma'$ will result.  If any of the added edges were from $E\setminus E^*$, then we would have $\ell_{\eta^*}(\gamma')<\ell_{\eta^*}(\gamma)$, which contradicts Lemma~\ref{lem:supp-mu-min-eta}.  Thus, if we add any edge in $E\setminus E^*$ to $\gamma\setminus E^*$ we must create a cycle.  This implies that $\gamma\setminus E^*$ restricted to any connected component of $G_{\overline{E^*}}$ is a spanning tree.  From Lemma~\ref{lem:M-Q} it follows that $|\gamma\cap E^*|=Q(G_{\overline{E^*}})-1=\mathcal{M}(E^*)$.
\end{proof}

\begin{lemma}\label{lem:eta-star-max}
The largest value $\eta^*$ attains is
\begin{equation*}
\eta^*_{\max} = \theta(E^*).
\end{equation*}
\end{lemma}
\begin{proof}
This follows from the fact that
\begin{equation*}
\begin{split}
|E^*|\eta^*_{\max} &= \sum_{e\in E^*}\eta^*(e) =
\sum_{e\in E^*}\sum_{\gamma\in\Gamma}\mathbbm{1}_{e\in\gamma}\mu^*(\gamma)\\
&= \sum_{\gamma\in\Gamma}\mu^*(\gamma)\sum_{e\in E^*}\mathbbm{1}_{e\in\gamma} =
\sum_{\gamma\in\supp\mu^*}\mu^*(\gamma)|\gamma\cap E^*| = \mathcal{M}(E^*),
\end{split}
\end{equation*}
where the final equality follows from Lemma~\ref{lem:Estar-min-isect}.
\end{proof}

Now we are ready to prove the first theorem that will allow us to use graph vulnerability to compute spanning tree modulus.

\begin{theorem}\label{thm:crit-set-eta-star}
Let $J\subseteq E$ be a critical set for $G$.  Then, for all $e\in J$,
\begin{equation*}
\eta^*(e) = \theta(J) = \theta(G)
= \theta(E^*) = \eta^*_{\max}.
\end{equation*}
\end{theorem}
\begin{proof}
Note that, similar to the proof of the previous theorem,
\begin{equation*}
|J|\eta^*_{\max} \ge \sum_{e\in J}\eta^*(e) = 
\sum_{\gamma\in\Gamma}\mu^*(\gamma)|\gamma\cap J|
\ge \mathcal{M}(J).
\end{equation*}
Thus, by Lemma~\ref{lem:eta-star-max},
\begin{equation*}
\theta(G) \ge \theta(E^*) =
\eta^*_{\max}  \ge
\frac{1}{|J|}\sum_{e\in J}\eta^*(e)
\ge \theta(J) = \theta(G),
\end{equation*}
and the inequalities are actually satisfied as equalities.  Since the average of $\eta^*$ over $J$ equals to $\eta^*_{\max}$, it follows that $\eta^*$ attains this maximum on all edges of $J$.
\end{proof}

Theorem~\ref{thm:crit-set-eta-star} provides the first step in an algorithm for computing the spanning tree modulus.  If we have an algorithm, such as the one provided by Cunningham, that finds both $\theta(G)$ and a critical set $J$, then we immediately know that $\eta^*$ takes the value $\theta(G)$ on the edges of $J$.  The theorem implies that $J\subseteq E^*$, but the inclusion could be strict.  In any case, this step provides the value of $\eta^*$ on at least one edge.  Next, we show how this procedure can be repeated recursively to find $\eta^*$ on the remaining edges.  To see this, we need to establish a few facts about a critical set $J$ and about the components of $G_{\overline{J}}$.

As described at the beginning of this section, the removal of $J$ from $G$ induces a set of subgraphs $\{G_i=(V_i,E_i)\}$ of $G$.  By construction, $E\setminus\left(\cup_iE_i\right)\subseteq J$.  However, there is, a priori, no reason to suppose that $(\cup_iE_i)\cap J$ is empty.  This is addressed in the following lemma.

\begin{lemma}\label{lem:J-crit-nonisect}
Let $J$ be a critical set for $G$ and let $G_i=(V_i,E_i)$ for $i=1,2,\ldots,k$ be the associated $k=Q(G_{\overline{J}})$ subgraphs constructed as above.  Then
\begin{equation*}
    \left(\bigcup_{i=1}^kE_i\right)\cap J = \emptyset.
\end{equation*}
\end{lemma}
\begin{proof}
Suppose the intersection is nonempty, and let
\begin{equation*}
J' = J\setminus\left(\bigcup_{i=1}^kE_i\right).
\end{equation*}
By assumption, $Q(G_{\overline{J}})=Q(G_{\overline{J'}})$ and $|J'|<|J|$.  Thus, by Lemma~\ref{lem:M-Q},
\begin{equation*}
\theta(J') = \frac{\mathcal{M}(J')}{|J'|}
> \frac{\mathcal{M}(J)}{|J|} = \theta(G),
\end{equation*}
which is a contradiction.
\end{proof}

Our next step is to show that any optimal pmf $\mu^*$ on $G$ induces marginal pmfs on the families $\Gamma_i$ of spanning trees on the subgraphs $G_i$.

\begin{lemma}\label{lem:mu-star-restricts}
Let $\mu^*$ be an optimal pmf for spanning tree modulus on $G$, let $J$ be a critical set, and let $G_i=(V_i,E_i)$ be one of the subgraphs of $G$ induced by $J$.  Define
\begin{equation*}
\mu_i(\gamma') := 
\mu^*\left(
\{\gamma\in\Gamma : \gamma\cap E_i = \gamma'\}
\right)
\quad
\text{for all }\gamma'\in\Gamma_i.    
\end{equation*}
Then $\mu_i\in\mathcal{P}(\Gamma_i)$.
\end{lemma}

\begin{proof}
By definition, $\mu_i$ is a non-negative function on $\Gamma_i$.  To see that it is a pmf, then, it remains to verify that $\mu_i(\Gamma_i)=1$.  Let $\mathbbm{1}_{\gamma\cap E_i=\gamma'}$ be the indicator variable
\begin{equation*}
\mathbbm{1}_{\gamma\cap E_i=\gamma'} :=
\begin{cases}
1 & \text{if }\gamma\cap E_i=\gamma',\\
0 & \text{if }\gamma\cap E_i\ne\gamma',
\end{cases}
\end{equation*}
and note that
\begin{equation*}
\sum_{\gamma'\in\Gamma_i}\mu_i(\gamma')
= \sum_{\gamma'\in\Gamma_i}\sum_{\gamma\in\Gamma}
\mathbbm{1}_{\gamma\cap E_i=\gamma'}\mu^*(\gamma) =
\sum_{\gamma\in\Gamma}\mu^*(\gamma)
\left(
\sum_{\gamma'\in\Gamma_i}
\mathbbm{1}_{\gamma\cap E_i=\gamma'}
\right).
\end{equation*}
Since $\gamma\cap E_i$ can match at most one of the trees in $\Gamma_i$, the innermost sum is at most one.  To see that it actually equals to one, we need to verify that $\gamma\cap E_i\in\Gamma_i$ for any $\gamma\in\supp\mu^*$.  That is, we need to show that every tree in the support of $\mu^*$ restricts as a tree on $G_i$.

Suppose, to the contrary, that there is a spanning tree $\gamma\in\supp\mu^*$ such that $\gamma\cap E_i$ does not form a spanning tree of $G_i$.  Then, since the set $\gamma\cap E_i$ is acyclic, either it doesn't span all of $G_i$ or it is a forest with more than one tree. In either case, there must exist an edge $e\in E_i$ with the property that $(\gamma\cap E_i)\cup\{e\}$ does not contain a cycle.  By Lemma~\ref{lem:J-crit-nonisect}, $e\notin J$.  Consider the set $\gamma'=\gamma\cup\{e\}$.  Since $\gamma$ is a spanning tree of $G$, $\gamma'$ must contain a cycle.  Moreover, since $\gamma'\cap E_i$ does not contain a cycle, the cycle in $\gamma'$ must cross at least one edge $e'$ of $J$.  By removing this edge from $\gamma'$, we obtain a new spanning tree $\gamma''=\gamma'\setminus\{e'\}\in\Gamma$ with the property that
\begin{equation*}
|\gamma''\cap J| < |\gamma\cap J| = \mathcal{M}(J),
\end{equation*}
which is a contradiction.  This shows that $\mu_i$ is indeed a marginal pmf on $\Gamma_i$.
\end{proof}

Next we show a technique for modifying pmfs on $\Gamma$ in such a way that the edge usage probabilities change only on $E_i$.

\begin{lemma}\label{lem:glue-pmfs}
Let $\mu^*$ be an optimal pmf for spanning tree modulus on $G$, let $J$ be a critical set, let $G_i=(V_i,E_i)$ be one of the induced subgraphs, and let $\mu_i$ be any pmf on $\Gamma_i$, the family of spanning trees of $G_i$.  Let $\eta^*$ and $\eta_i$ be the corresponding edge usage probabilities of $\mu^*$ and $\mu_i$.  Then there exists a pmf $\mu\in\mathcal{P}(\Gamma)$ with edge usage probabilities $\eta$ satisfying
\begin{equation*}
\eta(e) = 
\begin{cases}
    \eta_i(e) &\text{if }e\in E_i,\\
    \eta^*(e) &\text{if }e\notin E_i.
\end{cases}    
\end{equation*}
\end{lemma}

\begin{proof}
Consider the following procedure for randomly constructing a subset of edges $F\subset E$.  First, pick a random $\gamma\in\Gamma$ according to the pmf $\mu^*$.  Next, pick a random $\gamma_i\in\Gamma_i$ according to the pmf $\mu_i$.  Finally, define the random set $F$ as
\begin{equation*}
    F := (\gamma\setminus E_i)\cup\gamma_i.
\end{equation*}
We claim that $F\in\Gamma$.

To see this, recall that in the proof of Lemma~\ref{lem:mu-star-restricts} we showed that $\gamma\cap E_i\in\Gamma_i$.  This implies that $|\gamma\cap E_i|=|V_i|-1=|\gamma_i|$.  Since $\gamma_i\subset E_i$, it follows that 
\begin{equation*}
|F| = |\gamma| - |\gamma\cap E_i| + |\gamma_i| = |\gamma| = |V|-1,
\end{equation*}
so $F$ has the correct number of edges for a spanning tree.  Thus, $F$ can only fail to be a spanning tree if it contains a cycle.

Suppose such a cycle, $C\subset E$, does exist.  Since $F\cap E_i=\gamma_i$ is a spanning tree, $C$ must use at least one edge $e\in E\setminus E_i$.  Let $S$ be the largest connected subset of $C\setminus E_i$ containing $e$.  Since $S\in\gamma$, $S$ must be the edges traversed by a path connecting two distinct vertices $x,y\in V_i$ and using only edges in $E\setminus E_i$.  Since $\gamma\cap E_i$ must be a spanning tree of $G_i$ by Lemma~\ref{lem:mu-star-restricts}, it contains a path from $x$ to $y$ consisting only of edges in $E_i$.  But this implies that $\gamma$ contains two different paths connecting $x$ to $y$, which contradicts the fact that $\gamma$ is a spanning tree.

Hence, the procedure outlined above produces a random spanning tree on $G$.  Let $\underline{\gamma}$, $\underline{\gamma_i}$, and $\underline{F}$ represent the corresponding random objects described in this procedure.  As stated, $\underline{F}$ is a random spanning tree on $G$ with a corresponding pmf $\mu$.  Although the formula for $\mu$ is complex, the resulting edge usage probabilities, $\eta$, are straightforward to compute.

First, suppose that $e\in E_i$.  Then
\begin{equation*}
\eta(e) = \mathbb{P}_{\mu}(e\in\underline{F})
= \mathbb{P}_{\mu_i}(e\in\underline{\gamma_i}) = \eta_i(e).
\end{equation*}
Similarly, if $e\notin E_i$, then
\begin{equation*}
\eta(e) = \mathbb{P}_{\mu}(e\in\underline{F})
= \mathbb{P}_{\mu^*}(e\in\underline{\gamma}\setminus E_i) =
\mathbb{P}_{\mu^*}(e\in\underline{\gamma}) =
\eta^*(e).
\end{equation*}
\end{proof}

The recursive algorithm for spanning tree modulus is made possible by the following theorem.

\begin{theorem}\label{thm:eta-star-restricts}
Let $J$ be a critical set for $G$ and let $G_i=(V_i,E_i)$ for $i=1,2,\ldots,k$ be the associated $k=Q(G_{\overline{J}})$ subgraphs constructed as above.  Let $\eta^*$ be the optimal edge usage probabilities for $G$ and let $\eta^*_i$ be the optimal edge usage probabilities for each $G_i$.  Then, for every $e\in E_i$, $\eta^*(e) = \eta^*_i(e)$.
\end{theorem}
\begin{proof}
Let $\mu^*$ be an optimal pmf for the spanning tree modulus of $G$ and let $\mu_i^*$ be an optimal pmf for the spanning tree modulus of $G_i$.  Let $\mu\in\mathcal{P}(\Gamma)$ be the pmf guaranteed by Lemma~\ref{lem:glue-pmfs} and let $\eta$ be the corresponding edge usage probabilities.  since $\mu^*_i$ is optimal, Lemma~\ref{lem:mu-star-restricts} implies that
\begin{equation*}
\sum_{e\in E_i}\eta_i^*(e)^2 \le 
\sum_{e\in E_i}\eta^*(e)^2.
\end{equation*}
By Lemma~\ref{lem:glue-pmfs}, then,
\begin{equation*}
\sum_{e\in E}\eta^*(e)^2 \le 
\sum_{e\in E}\eta(e)^2 = \sum_{e\in E_i}\eta^*_i(e)^2
+ \sum_{e\in E\setminus E_i}\eta^*(e)^2
\le \sum_{e\in E}\eta^*(e)^2.
\end{equation*}
By part 2 of Theorem~\ref{thm:mod-summary}, it follows that $\eta=\eta^*$.  In particular, for $e\in E_i$,
\begin{equation*}
\eta^*(e) = \eta(e) = \eta_i^*(e).
\end{equation*}
\end{proof}

Theorem~\ref{thm:eta-star-restricts} suggests the following algorithm for computing spanning tree modulus using an algorithm for graph vulnerability, such as the one Cunningham provides.  First, find the value $\theta(G)$ and a critical set $J$.  Theorem~\ref{thm:crit-set-eta-star} shows that $\eta^*$ takes the value $\theta(G)$ on all edges of $J$.  Now, remove these edges from $G$.  This results in a number of connected subgraphs $G_1,G_2,\ldots$.  By Theorem~\ref{thm:eta-star-restricts}, if we find $\theta(G_1)$ along with a critical set $J_1$ of $G_1$, then the optimal $\eta^*$ (for spanning tree modulus of $G$!) takes the value $\theta(G_1)$ on the edges of $J_1$.  This procedure can be repeated, each time finding $\eta^*$ for at least one edge.  Thus, after finitely many iterations, $\eta^*$ will be known for all edges.  A more detailed description of this algorithm is presented in Section~\ref{sec:spt-mod-algorithm}.

\section{Cunningham's algorithm for graph vulnerability}
\label{sec:graph-vulnerability}

In this section, we review the algorithm presented in~\cite{cunningham1985optimal} along with its theoretical background.  In some places we have provided our own explanations or proofs where we feel it will aid in understanding.

First, if $J\subseteq E$ and $\gamma\in\Gamma$, then from~\eqref{eq:theta-J}
\begin{equation*}
\mathcal{M}(J) \le |J\cap\gamma| \le \min\{|\gamma|,|J|\} = 
\min\{|V|-1,|J|\}.
\end{equation*}
Moreover, since $J \subseteq E$, we have $|J| \le |E|$.
Thus, for any $\emptyset\ne J\subseteq E$,
\begin{equation}\label{eq:Theta}
\theta(J) = \frac{\mathcal{M}(J)}{|J|} \in
\Theta := 
\left\{
\frac{p}{q} :
1\le p\le \min\{|V|-1,q\},\;1\le q\le |E|
\right\}.
\end{equation}
Since there are finitely many options for the numerator and denominator in $\theta(J)$, there are finitely many possible values for $\theta(G)$.  In~\cite{cunningham1985optimal}, this fact is exploited to compute $\theta(G)$.  Essentially, one performs a binary search on the set of possible values, making use of an oracle that can determine whether or not $\theta(G)\le\frac{p}{q}$ for a given fraction $\frac{p}{q}$.  This oracle is derived by performing a few transformations on the problem that convert it to a problem on a matroid that can be solved by a greedy algorithm.

\begin{lemma}
For any $p,q\ge 1$, we have 
\begin{equation}\label{eq:vulnerability equivalency}
  \theta(G) \le \frac{p}{q} \iff 0 \le \min_{J\subseteq E}\Big(\frac{p}{q}|J| - \mathcal{M}(J)\Big). 
\end{equation}
\end{lemma}

\begin{proof}
\begin{equation*}
\begin{split}
 \theta(G) \le \frac{p}{q} &\iff \max_{J \subseteq E} \theta(J) \le \frac{p}{q}  \iff \theta(J) \le \frac{p}{q}\quad \forall J \subseteq E    \\
 & \iff 0 \le \frac{p}{q}|J| - \mathcal{M}(J) \quad \forall J \subseteq E    \\
 & \iff 0 \le \min_{J\subseteq E}\Big(\frac{p}{q}|J| - \mathcal{M}(J)\Big).
\end{split}
\end{equation*}

\end{proof}

 From this lemma, we can see that checking whether $\theta(G)\le\frac{p}{q}$ is equivalent to solving the minimization problem in~\eqref{eq:vulnerability equivalency}.  Cunningham called this minimization problem the \textit{optimal attack problem}.

\subsection{Vulnerability and the graphic matroid}

Next, the optimal attack problem is converted into a question about a matroid.  By a \emph{matroid}, we mean (one of several equivalent definitions) a tuple $M=(E,f)$ where $E$ is a finite set and $f:2^E\to\mathbb{N}_0$ is a function assigning to every subset of $E$ a non-negative integer value satisfying the following properties.
\begin{enumerate}
\item For each $J \subseteq E$, $f(J) \leq |J|$.
\item If $J' \subseteq J \subseteq E$, then $f(J') \leq f(J)$.
\item If $J, J' \subseteq E$, then $f(J \cup J') + f(J \cap J') \leq f(J) + f(J')$.
\end{enumerate}
Any function $f$ that satisfies these conditions is called a \emph{polymatroid function}.  An important example of a matroid in the context of the present work is the \emph{graphic matroid} $M=(E,f)$ where $E$ is the edge set of a connected graph, and $f$ is the \emph{graphic rank function} defined as
\begin{equation}\label{eq:graphic-rank}
f(J) := |V| - Q(G_J).    
\end{equation}

Using Lemma~\ref{lem:M-Q}, we can write
\begin{equation}\label{eq:MJ-fJbar}
\mathcal{M}(J) = |V|-1-f(\overline{J}).
\end{equation}
So, the minimization~\eqref{eq:vulnerability equivalency} can be rewritten as
\begin{equation*}
    \min_{J\subseteq E}\Big(\frac{p}{q}|J| - |V| + 1  + f(\overline{J})\Big). 
\end{equation*}
Thus, we arrive at a fundamental lemma for Cunningham's algorithm.
\begin{lemma}\label{lem:fund-cunningham}
The following are equivalent.
\begin{enumerate}
    \item $\theta(G)\le\frac{p}{q}$
    \item $\min\limits_{J\subseteq E}\left(\frac{p}{q}|J|+f(\overline{J})\right) \ge |V|-1$
\end{enumerate}
\end{lemma}
As shown in~\cite{cunningham1985optimal}, the minimum value on the left-hand side of 2 is related to the concept of polymatroid bases.

\subsection{The polymatroid theorem and graph vulnerability}

Given a polymatroid function $f$ on a set $E$, we 
define the \emph{polymatroid} associated with $f$ to be the polytope defined as
\begin{equation}\label{eq:Pf}
    P(f) := \{ x \in \mathbb{R}_{\ge 0}^{E},~ x(J) \leq f(J)\;\text{for all } J \subseteq E \},
\end{equation}
where we use the notation $x(J)$ to represent the sum
\begin{equation*}
x(J) = \sum_{e\in J}x(e).
\end{equation*}
For any $y \in \mathbb{R}_{\ge 0}^{E}$, any maximal vector $x \in P(f)$ with $x\leq y$ is called a $P(f)$-basis of $y$.

The following is the key theorem of~\cite{cunningham1985optimal} that establishes the connection between graph vulnerability and polymatroid bases.
\begin{theorem}[Theorem~1 of~\cite{cunningham1985optimal}]\label{thm:polymatroid theorem}
Let $y \in \mathbb{R}_{\ge 0}^{E}$ and let $x$ be any $P(f)$-basis of $y$, then
\begin{equation}\label{eq:basis}
    x(E) = \min_{J \subseteq E}\big(y(J) + f(\Bar{J}) \big).
\end{equation}
\end{theorem}

To see how Theorem~\ref{thm:polymatroid theorem} can be applied to the graph vulnerability problem, we recall the graphic rank function $f$ defined in~\eqref{eq:graphic-rank} is a polymatroid function.  Thus, if we define $y\equiv\frac{p}{q}$ (the constant vector) and let $x$ be a $P(f)$-basis of $y$, then Lemma~\ref{lem:fund-cunningham} combined with Theorem~\ref{thm:polymatroid theorem} implies that $\theta(G)\le\frac{p}{q}$ if and only if $x(E)\ge |V|-1$.  Thus, an algorithm that can efficiently compute polymatroid bases can be used together with a binary search to compute the vulnerability of a graph.  Cunningham provided such an algorithm in~\cite{cunningham1985optimal}.

\subsection{Cunningham's greedy algorithm}\label{sec:cunningham-greedy}

By definition, if $f$ is a polymatroid function, $x\in P(f)$, $x\le y$ and $J\subseteq E$, then $x(\overline{J})\le f(\overline{J})$ and $x(J)\le y(J)$, so
\begin{equation}\label{eq:pre-cunningham}
    x(E) = x(J) + x(\overline{J})
    \le y(J) + f(\overline{J}).
\end{equation}
Theorem~\ref{thm:polymatroid theorem} implies that if $x$ is a $P(f)$-basis of $y$, then there exists a $J\subseteq E$ for which~\eqref{eq:pre-cunningham} holds as equality.  In particular, this $J$ must satisfy
\begin{equation}\label{eq:Pf-basis}
x = y\;\text{on }J\qquad\text{and}\qquad
x(\overline{J}) = f(\overline{J}).
\end{equation}
A set $\overline{J}$ satisfying the second of these equalities is called \emph{tight} with respect to $x$.  Cunningham's greedy algorithm simultaneously finds a $P(f)$-basis of a given $y$ together with a set $J$ satisfying~\eqref{eq:Pf-basis}.  A pseudocode listing of this algorithm is shown in Algorithm~\ref{alg:Cunningham}.

\begin{algorithm}
\caption{Cunningham's algorithm\\
Input: G=(V,E), y}\label{alg:Cunningham}
\begin{algorithmic}[1]
	\State $x \gets 0$ 
	\State $\overline{J} \gets \emptyset$
	\ForAll{$j\in E$}
	\State $\epsilon, J'(j) \gets \min \{f(J')-x(J') : j \in J'\subseteq E\}$
	\If
	{$\epsilon < y(j) - x(j)$}
	\State $\overline{J} \gets \overline{J} \cup J'(j)$
	\Else
	{}
	\State $\epsilon = y(j) -x(j)$
	\EndIf
	\State $x(j) \gets x(j) + \epsilon$
	\EndFor
	\State\Return $x, J=E\setminus\overline{J}$
\end{algorithmic}
\end{algorithm}

\begin{remark}
In fact, Cunningham's original version of Algorithm~\ref{alg:Cunningham} used the non-strict inequality ($\le$) rather than the strict inequality ($<$) on line 5.  This choice only exacerbates the difficulty of finding the critical set described in Section~\ref{sec:obtain-crit-set}.  Additional modifications are described in Section~\ref{sec:cunningham-mods}.
\end{remark}

At the beginning of the algorithm, $x$ is initialized to the zero vector and $\overline{J}$ is initialized to the empty set (in other words, $J$ is initialized to all of $E$).  The algorithm proceeds by looping over all edges $j$ of $E$.  On each iteration of the loop, one computes the value
\begin{equation}\label{eq:epsilon max}
    \epsilon_{\max}(j) = \max\{\epsilon: x + \epsilon\textbf{1}_j \in P(f)\}
\end{equation}
along with a set $J'(j)$ containing $j$ and satisfying
\begin{equation}\label{eq:tight-set-j}
f(J'(j)) = x(J'(j)) + \epsilon_{\max}(j).
\end{equation}
The value $\epsilon_{\max}$ is the largest value by which we can increment $x$ on edge $j$ without leaving $P(f)$ and the set $J'(j)$ is a constraint that would be violated if we were to increase $x$ by any more on that edge.

Next in the algorithm is a decision.  If $x(j)$ can be increased by $\epsilon_{\max}$ without violating the constraint that $x(j)\le y(j)$, then $x(j)$ is incremented by that amount and the edges in $J'(j)$ are added to $\overline{J}$ (equivalently, $J$ is intersected with $\overline{J'(j)}$).  Otherwise, $x(j)$ is set to $y(j)$ and $\overline{J}$ is not changed.

A proof of the correctness of the algorithm can be found in~\cite{cunningham1985optimal}.  Additional details omitted from the discussion there can be found in~\cite{kottegoda2020}.  The important theorem is as follows.
\begin{theorem}\label{thm:cunningham-alg}
When Algorithm~\ref{alg:Cunningham} completes, the vector $x$ is a $P(f)$-basis of $y$ and $\overline{J}$ is a \emph{tight set} with respect to $x$.  That is, $x(\overline{J})=f(\overline{J})$.
\end{theorem}

\subsection{Minimum cut formulation of the subproblem}\label{sec:min-cut}

The most computationally challenging part of Algorithm~\ref{alg:Cunningham} occurs on line 4 when minimizing $f(J)-x(J)$ over all subsets $J\subseteq E$ that contain a particular edge $j$.  Although this appears to be combinatorially difficult at first look, Cunningham's approach was to recast this optimization problem as a minimum cut problem on an auxiliary graph.

As described in~\cite{cunningham1985optimal}, given $x \in P(f)$ and $j \in E$, the first step is to construct an undirected capacitated graph $G'$.  The vertex set for $G'$ is $V \cup \{r, s\}$, where $r$ and $s$ are new vertices which will respectively be the source and sink of a network flow. Each edge $e \in E$ has a corresponding edge in $G'$ with capacity $\frac{1}{2}x(e)$. Edges are added connecting from $s$ to each vertex $v \in V$ and having capacity 1. Edges are also added from $r$ to each $v \in V$ having capacity $\infty$ if $v$ is an end point of $j$, or $\frac{1}{2}x(\delta(v))$ otherwise. Here $\delta(U) = \{ e \in E,~ e \text{ has exactly one end in}~ U \subseteq V \}$, and $\delta (v)$ is shorthand for $\delta(\{v\})$.  An example of this augmented graph $G'$ is shown in Figure~\ref{fig:graph-cut1}.  As shown in~\cite{cunningham1985optimal}, it is straightforward to recover $\epsilon_{\max}(j)$ in~\eqref{eq:epsilon max} from the value of a minimum $rs$-cut in $G'$.  Moreover, the edges of a minimum $rs$-cut (after removing any edges incident on $r$ or $s$) form a tight set $J'(j)$ satisfying~\eqref{eq:tight-set-j}.

\begin{figure}
	\begin{center}
		\includegraphics[width=0.8\textwidth]{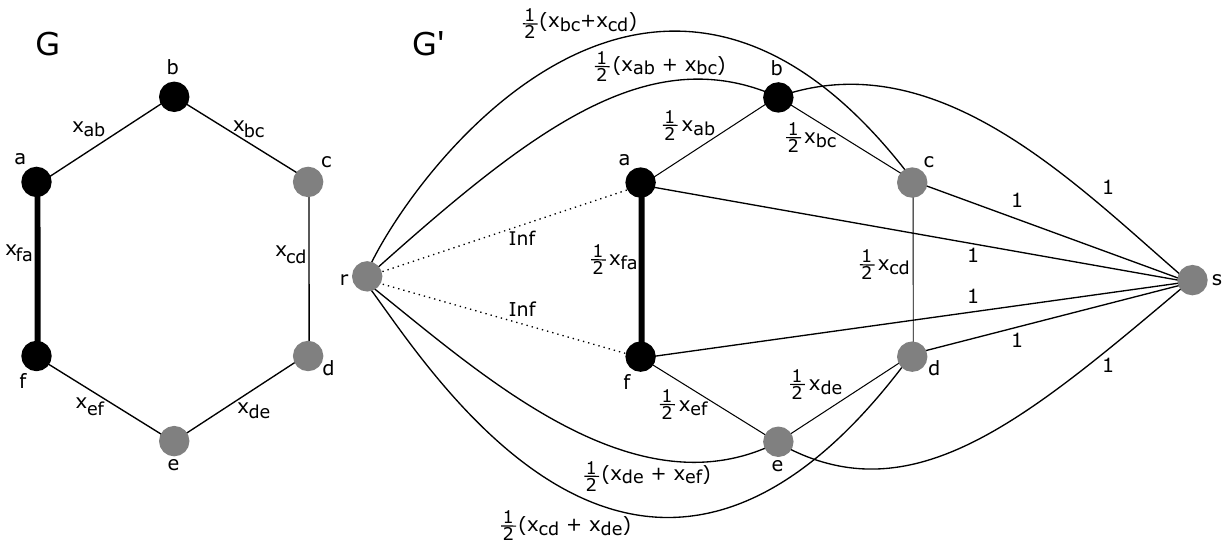}
		\caption{An example of the minimum cut problem described in Section~\ref{sec:min-cut}.  The original graph $G$ is shown on the left.  The goal is to increase the value of $x$ on the highlighted edge $j=\{f,a\}$ as much as possible without violating any constraints of the polymatroid defined in~\eqref{eq:Pf}.  The right-hand side shows the corresponding flow graph $G'$.  The maximum increment $\epsilon_{\max}(j)$ is computed from the value of the minimum cut; a tight set $J'(j)$ is obtained from the minimum cut edges.}
	\label{fig:graph-cut1}
	\end{center}
\end{figure}

\subsection{A polynomial-time algorithm for graph vulnerability}

The polynomial-time algorithm for computing $\theta(G)$, described in~\cite{cunningham1985optimal} proceeds as follows.  First, recall that $\theta(G)$ must belong to a finite set (e.g., $\theta(G)$ must belong to $\Theta$ as defined in~\eqref{eq:Theta}).  Thus, if one can produce a polynomial-time oracle to determine whether $\theta(G)\le\frac{p}{q}$ for a given $\frac{p}{q}$, then a simple binary search of $\Theta$ will produce the value $\theta(G)$ in polynomial time.  Algorithm~\ref{alg:Cunningham} provides just such an oracle if one uses the minimum cut formulation of Section~\ref{sec:min-cut}.

\section{Modifications to Cunningham's algorithm}\label{sec:cunningham-mods}

Before we proceed to computing the spanning tree modulus, we need to discuss some modifications to Cunningham's algorithm.  The first modification is necessary in order to ensure that we can obtain not only the graph vulnerability, $\theta(G)$ but also a critical set of edges.  The second modification is helpful for efficient computation in exact arithmetic.

\subsection{Obtaining a critical set}\label{sec:obtain-crit-set}

An important step in the spanning tree modulus algorithm described in Section~\ref{sec:vulnerability-to-modulus} is the identification of both the vulnerability $\theta(G)$ \emph{and} a critical set $J$ satisfying~\eqref{eq:critical-set}.  This latter point proved somewhat more difficult than we originally anticipated, due to an interesting consequence of the theory.

What we observed in our initial implementations of Algorithm~\ref{alg:Cunningham} is that, when it is called with $y=\theta(G)$, it will often return an empty set $E\setminus\overline{J}$.  (That is, at the end of the algorithm, $\overline{J}=E$.)  This phenomenon is more common when using exact arithmetic than when using floating point arithmetic.  This failure to produce a critical set initially surprised us and eventually exposed a subtle misunderstanding on our part.  Since it does not appear to be addressed in the literature, we wish to call attention to this aspect of the algorithm and to present a simple modification that is guaranteed to produce a critical set.

The key to understanding why Algorithm~\ref{alg:Cunningham} might return an empty set when called with $y$ exactly equal to $\theta(G)$ is to recall the guarantee on $\overline{J}$.  In particular, Theorem~\ref{thm:cunningham-alg} guarantees that $\overline{J}$ is tight with respect to the $P(f)$-basis $x$.  Now, consider the following theorem.

\begin{theorem}\label{thm:cunningham-problem}
If $\theta(G)=\frac{p}{q}$ and if $x$ is a $P(f)$-basis for $y\equiv\frac{p}{q}$, then $E$ is tight with respect to $x$.
\end{theorem}
\begin{proof}
Since $\theta(G)=\frac{p}{q}$, we have
\begin{equation*}
\begin{aligned}
f(E) &\ge x(E)
&& \text{since $x\in P(f)$}\\
&= \min_{J\subseteq E}\left(\frac{p}{q}|J|+f(\overline{J})\right)\qquad
&& \text{by Theorem~\ref{thm:polymatroid theorem}}\\
& \ge |V|-1 && \text{by Lemma~\ref{lem:fund-cunningham}}\\
&= f(E) && \text{since $G$ is connected}.
\end{aligned}
\end{equation*}
This implies that $x(E)=f(E)$.  In other words, $E$ is tight.
\end{proof}

\begin{figure}
    \centering
    \includegraphics[width=0.4\textwidth]{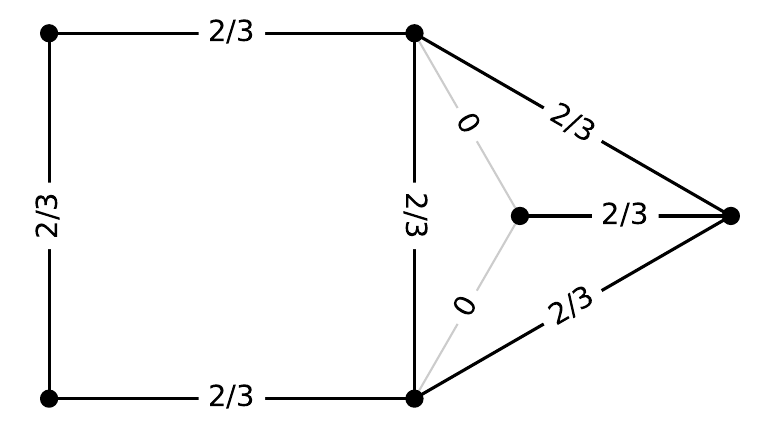}\hspace{1cm}
    \includegraphics[width=0.4\textwidth]{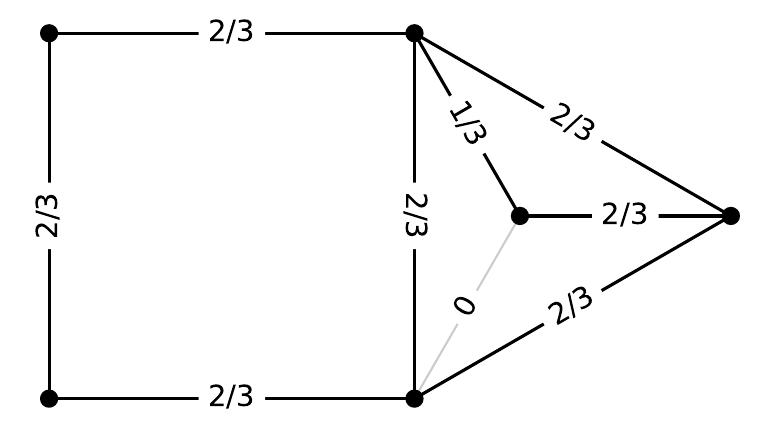}
    \caption{An example, described in Section~\ref{sec:obtain-crit-set}, of how Cunningham's algorithm may not produce a critical set of edges.}
    \label{fig:cunningham-problem}
\end{figure}

Theorem~\ref{thm:cunningham-problem} shows that there is no theoretical barrier to Algorithm~\ref{alg:Cunningham} returning the empty set when $y=\theta(G)$.  To see how this can occur in practice, consider the graph in Figure~\ref{fig:cunningham-problem}.  This graph has $\theta(G)=\frac{2}{3}$ with a critical set consisting of the three left-most edges.  The left side of Figure~\ref{fig:cunningham-problem} shows the value of $x$ after 7 steps of Algorithm~\ref{alg:Cunningham} with $y\equiv\frac{2}{3}$.  Edges are labeled with their current $x$-value and the darker edges are those that have already been visited by the algorithm.  In each of the steps thus far, the inequality on line 5 of the algorithm has been false and, therefore, $\overline{J}=\emptyset$.  There are two edges remaining to process.

Suppose the algorithm next moves to the upper of these two edges.  In considering the polymatroid constraints in line 4 of the algorithm, there are two possible ``tightest'' constraints that might be located.  The first is the set $J'$ consisting of the tetrahedral subgraph.  This set has rank $f(J')=3$ and (currently) $x(J')=\frac{8}{3}$.  Thus, $x$ on the edge currently under consideration cannot be made larger than $\frac{1}{3}$.  The other constraint that may be found in this step is the set $J'=E$.  This set has rank $f(J')=5$ and (currently) $x(J')=\frac{14}{3}$.  Again, $x$ on the edge currently under consideration cannot be increased above $\frac{1}{3}$.  The latter choice will result in $\overline{J}=E$ and the algorithm will not locate a critical set.

Which of these two constraints are found on line 4 of the algorithm depends heavily on implementation details for the minimum cut algorithm.  For example, we have observed that the algorithm returns an empty set more frequently in our Python implementation (based on the NetworkX library~\cite{Networkx}) than in our C++ implementation (based on the Boost Graph Library~\cite{siek2002boost}) of this algorithm.

In practice, this means that using the oracle-based binary search described in~\cite{cunningham1985optimal,gueye2010design} will always find $\theta(G)$, but may not result in finding a critical edge set.  Our suggestion to address this problem is based on the following theorem.

\begin{theorem}
Suppose that $\theta(G)=\frac{p}{q}$ and  that $\frac{p'}{q'}<\frac{p}{q}$ satisfies
\begin{equation}\label{eq:pq-only-choice}
    \left(\frac{p'}{q'},\frac{p}{q}\right]
    \cap \Theta = \left\{\frac{p}{q}\right\}.
\end{equation}
Suppose further that $\tilde{x}$ is a $P(f)$-basis for $y\equiv\frac{p'}{q'}$ and that $\overline{J}$ is tight with respect to $\tilde{x}$.  Then $J$ is a critical set for $G$. 
\end{theorem}
\begin{proof}
First, we show that $\overline{J}$ cannot be $E$.  If it were, then we would have
\begin{equation*}
|V|-1 = f(E) = \tilde{x}(E)
= \min_{J'\subseteq E}\left\{\frac{p'}{q'}|J'|+f(\overline{J'})\right\},
\end{equation*}
which, by Lemma~\ref{lem:fund-cunningham}, would imply that
\begin{equation*}
\theta(G) \le \frac{p'}{q'} < \frac{p}{q} = \theta(G),
\end{equation*}
yielding a contradiction.

Now, when Cunningham's algorithm terminates, we have both $\tilde{x}$ and the tight set $\overline{J}$ with the property that $\tilde{x}=\frac{p'}{q'}$ on $J\ne\emptyset$.  Moreover, since $\theta(G)>\frac{p'}{q'}$, Lemma~\ref{lem:fund-cunningham} implies that
\begin{equation*}
|V|-1 > 
\frac{p'}{q'}|J|+f(\overline{J}).
\end{equation*}
By~\eqref{eq:MJ-fJbar}, this implies that
\begin{equation*}
\frac{p'}{q'}|J| < |V|-1 - f(\overline{J}) = \mathcal{M}(J).
\end{equation*}
Thus, we have that
\begin{equation*}
\frac{p'}{q'}<\frac{\mathcal{M}(J)}{|J|}
\le\theta(G) = \frac{p}{q}.
\end{equation*}
Choosing $p'$ and $q'$ so that~\eqref{eq:pq-only-choice} holds guarantees that $\mathcal{M}(J)/|J|=p/q$ and, thus, that $J$ is critical.
\end{proof}

Thus, if Cunningham's algorithm produces the $\theta(G)$, but fails to produce a critical set, we can re-run Cunningham's algorithm with $y$ slightly smaller than $\theta(G)$ to obtain a critical set.  A choice for $p'$ and $q'$ can be found using the following simple lower bound on the distance between distinct numbers in $\Theta$.  If $p/q, p'/q'\in\Theta$ and $p'/q'<p/q$, then
\begin{equation}
  \frac{p}{q} - \frac{p'}{q'} = \frac{pq' -qp'}{qq'} \geq \frac{1}{|E|^2}.
\end{equation}
Thus, choosing
\begin{equation*}
\frac{p'}{q'} = \frac{p}{q} - \frac{1}{|E|^2}
= \frac{p|E|^2 -q}{q|E|^2}
\end{equation*}
guarantees that~\eqref{eq:pq-only-choice} holds.

\subsection{Modifications for exact arithmetic}

Our primary goal in this section is to build upon Cunningham's algorithm to construct an algorithm for computing spanning tree modulus in exact arithmetic.  We first observe that this is already possible using Algorithm~\ref{alg:Cunningham} as written, provided all computations are performed using rational arithmetic.  The main downside to this approach is that the arithmetic (which must be done in software) is much slower than arithmetic performed on hardware.  On the other hand, if a floating point representation is used to improve the speed of the algorithm, one naturally sacrifices exactness for numerical approximation, and the impact of these approximations do not appear to be analyzed in the literature.  As an alternative, this section presents an implementation using integer arithmetic, which can be performed on hardware.

The modification for integer arithmetic is based on the observation that if $f$ is a polymatroid function and if $q>0$, then $qf$ is a polymatroid function.  Moreover, there is a straightforward connection between $P(f)$-bases and $P(qf)$-bases.

\begin{lemma}\label{lem:Pf-Pqf}
Let $x$ be a $P(f)$-basis of y. Then $qx$ is a $P(qf)$-basis of $qy$, for any $q>0$.
\end{lemma}

\begin{proof}
By definition,
\begin{equation}
\begin{split}
    P(qf) &= \{ x \in \mathbb{R}_{\ge 0}^{E}:~ x(J) \leq qf(J)\quad \forall J \subseteq E \}\\
    &= \{ x \in \mathbb{R}_{\ge 0}^{E}:~ \frac{1}{q}x(J) \leq  f(J)\quad \forall J \subseteq E \}\\
     &= \{ qx \in \mathbb{R}_{\ge 0}^{E}:~ x(J) \leq f(J)\quad \forall J \subseteq E \} = qP(f).
\end{split}
\end{equation}

Since $x$ is a $P(f)$-basis of $y$, we have $qx \leq qy$. Thus, in order to show that $qx$ is a $P(qf)$ basis of $qy$, we need only verify maximality.  If there exists $z \in P(qf)$ such that $z \leq q y$ and $ qx \leq z$ then $\frac{1}{q}z \in P(f)$, $\frac{1}{q}z \leq y$, and $x \leq \frac{1}{q}z$, which implies that $\frac{1}{q}z = x$ and, therefore, that $z = qx$.
\end{proof}

\begin{algorithm}
\caption{Integer arithmetic version of Cunningham's algorithm\\ Input: $G=(V,E), p, q$}\label{alg:Cunningham-integer}
\begin{algorithmic}[1]
	\State $x' \gets 0$ 
	\State $\overline{J} \gets \emptyset$
	\ForAll{$j \in E$}
	\State $\epsilon, J'(j) \gets \min \{qf(J')-x'(J') : j \in J'\subseteq E\}$
	\If
	{$\epsilon < p - x'(j)$}
	\State $\overline{J} \gets \overline{J} \cup J'(j)$
	\Else
	{}
	\State $\epsilon = p -x'(j)$
	\EndIf
	\State $x'(j) \gets x'(j) + \epsilon$
	\EndFor
		\State\Return $x', J = E\setminus\overline{J}$
\end{algorithmic}
\end{algorithm}

In light of Lemma~\ref{lem:Pf-Pqf}, consider Algorithm~\ref{alg:Cunningham-integer}.  Through a step-by-step comparison, one can see that  Algorithm~\ref{alg:Cunningham} (with $y=\frac{p}{q}$) and Algorithm~\ref{alg:Cunningham-integer} are equivalent in the sense that $x'=qx$ and both produce the same $\overline{J}$.  An important aspect of this comparison is the relationship between line 4 in each algorithm.  Lemma~\ref{lem:Pf-Pqf} shows that these two steps effectively compute the same thing, with $\epsilon$ differing only by the multiplicative constant $q$.

Note that $x'$ is initialized as an integer vector in line 1 of Algorithm~\ref{alg:Cunningham-integer}. Since $qf(J')$ is an integer for any $J'\subseteq E$, $\epsilon$ is also an integer in line 4. Since $p$ is an integer $\epsilon$ remains an integer even if line 8 is executed. Thus, in line 10, the updated $x'$ remains an integer vector. In each step described, only integer arithmetic is needed.  All that remains, then, is to verify that the computation on line 4 can be performed using integer arithmetic.

This is the consequence of a straightforward modification to the minimum cut algorithm described in Section~\ref{sec:min-cut} (Figure~\ref{fig:graph-cut1} may help visualize the argument).  If each capacity in the graph $G'$ is multiplied by $2q$, the resulting capacities are integer valued.  Edges which had capacity $1$ now have capacity $2q$; edges which had capacity $x(e)$ now have capacity $2x'(e)$; and edges which had capacity $\frac{1}{2}x(\delta(v))$ now have capacity $x'(\delta(v))$.  Moreover, any minimum $rs$-cut of $G'$ with the original capacities is also a minimum $rs$-cut with the new capacities and vice versa.  To perform line 4 of Algorithm~\ref{alg:Cunningham-integer}, then, we may generate $G'$ as before, but with all capacities multiplied by $2q$.  A minimum cut can now be found using an integer arithmetic maximum flow algorithm.  The edges of this cut are the edges $J'(j)$ that we seek and the value of the cut is $2\epsilon$ (twice the largest increment we can make to $x'(j)$ while remaining in $P(qf)$).  The value of $\epsilon$ can then be found by integer division.

\section{The spanning tree modulus algorithm}
\label{sec:spt-mod-algorithm}

As outlined in Section~\ref{sec:vulnerability-to-modulus}, the modified version of Cunningham's algorithm presented in Section~\ref{sec:cunningham-mods} allows us to compute the spanning tree modulus of $G$ in exact arithmetic.  Pseudocode for the algorithm is shown in Algorithm~\ref{alg:modulus}.

\begin{algorithm}
\caption{Modulus using graph vulnerability\\ Input: $G=(V,E)$}\label{alg:modulus}
\begin{algorithmic}[1]
	\State $q \gets \text{queue}([G])$
	\While{$q$ is not empty}
	\State $G'\gets \text{pop}(q)$
	\State compute $\theta(G')$ and a critical set $J$
	\ForAll{$e\in J$}
	\State $\eta^*(e)\gets\theta(G')$
	\EndFor
	\State remove edges in $J$ from $G'$
	\ForAll{nontrivial connected components $G''$ of $G'$}
	\State $\text{push}(q,G'')$
	\EndFor
	\EndWhile
\end{algorithmic}
\end{algorithm}

We initialize a queue of graphs with the initial graph $G$.  While this queue is not empty, we remove a graph $G'$ from the queue and compute its vulnerability along with a critical set of edges.  As described in Section~\ref{sec:vulnerability-to-modulus}, this tells us the value of the optimal $\eta^*$ on the edges of $J$.  Applying Theorem~\ref{thm:eta-star-restricts}, we then remove $J$ from $G'$ and add all nontrivial connected components of the resulting graph back into the queue for processing.  Once the queue is empty, we will have found the value of $\eta^*$ on all edges of $G$.

\subsection{Examples}

Here, we present a few examples of spanning tree modulus computed using the new algorithm.

\subsubsection{Small step-by-step example}
Consider the graph in Figure~\ref{fig:spt-algorithm} (Zachary's karate club network). The sequence of figures shows the order in which values of $\eta^*$ are discovered by Algorithm~\ref{alg:modulus}. The critical set for the original graph is the single edge highlighted in the first subfigure.  Since every spanning tree of $G$ must use this edge, its $\eta^*$ value is $1$. When this edge is removed indicated by the dashed line in the next subfigure, the graph is split into two components, only one of which is nontrivial.  Again, the critical set found by the algorithm is highlighted.  This time, the associated $\eta^*$ is $1/2$.  When these edges are removed, there is again a single nontrivial component, as shown in the third subfigure, and the algorithm proceeds by finding a critical set of this subgraph.  After a few more steps, $\eta^*$ is known on all edges and the algorithm terminates.  The modulus can be computed from $\eta^*$ as
\begin{equation*}
\Mod(\Gamma) = \left(
1 
+ 30\left(\frac{1}{2}\right)^2
+ 5\left(\frac{2}{5}\right)^2
+ 8\left(\frac{3}{8}\right)^2
+ 34\left(\frac{6}{17}\right)^2
\right)^{-1}
= \frac{680}{9969}.
\end{equation*}

\begin{figure}
	\begin{center}
		\includegraphics[width=\textwidth]{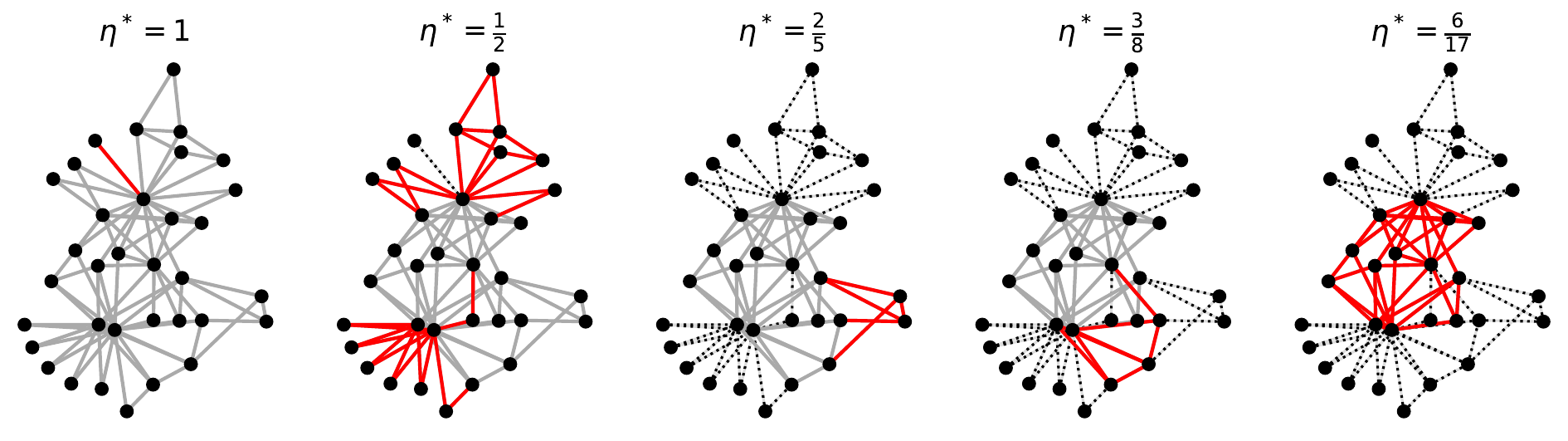}
		\caption{Steps of the spanning tree modulus applied to Zachary's karate club graph.}
	\label{fig:spt-algorithm}
	\end{center}
\end{figure}

\subsubsection{Timing experiments}\label{sec:timing}

Since $\eta^*$ is found for at least one edge in each iteration of Algorithm~\ref{alg:modulus}, the algorithm calls Cunningham's algorithm (Algorithm~\ref{alg:Cunningham}) $O(|E|)$ times. Since Algorithm~\ref{alg:Cunningham} completes in polynomial time, so does Algorithm~\ref{alg:modulus}. Figure~\ref{fig:time-complexity} demonstrates the efficiency of the method.  For this figure, we considered four classes of graphs, as described below.  For each class of graph, we generated a number of test cases on which we ran a C++ implementation of the algorithm to compute modulus.  Then, on a logarithmic scale, we performed a least-squares linear regression to find an approximate run time complexity of the form $c|E|^p$.  In all cases, the run-time appears to be no worse than order $O(|E|^3)$.  The four classes of graph used in this experiment are as follows.
\begin{enumerate}
    \item The complete graphs $K_n$ with $n=3,4,\ldots,40$.  For these graphs, the entire edge set is critical.
    \item A class of multipartite graphs parameterized by an integer $k\ge 2$.  The vertices of the graph are partitioned into $k$ sets $V_1,V_2,\ldots,V_k$ with $|V_i|=i$.  For $i=1,2,\ldots,k-1$, every vertex in $V_i$ is connected to every vertex in $V_{i+1}$.  This test included values of $k$ between $2$ and $16$.  These graphs tend to have optimal $\eta^*$ which take $k-1$ distinct values.
    \item Erd\"os--R\'enyi $G(n,p)$ graphs with $n$ chosen randomly in the range $[10,200]$ and $p=2\log(n)/n$.  These graphs tend have $\eta^*$ with only a few distinct values.
    \item Random geometric graphs formed by placing $n$ vertices in the unit square independently and uniformly at random and then connecting vertices with Euclidean distance closer than radius $r=3/\sqrt{n}$.  Values of $n$ were chosen uniformly from the range $[10,100]$.  These graphs tend to have $\eta^*$ that take a variety of distinct values.
\end{enumerate}

\begin{figure}
    \centering
    \includegraphics[width=0.5\textwidth]{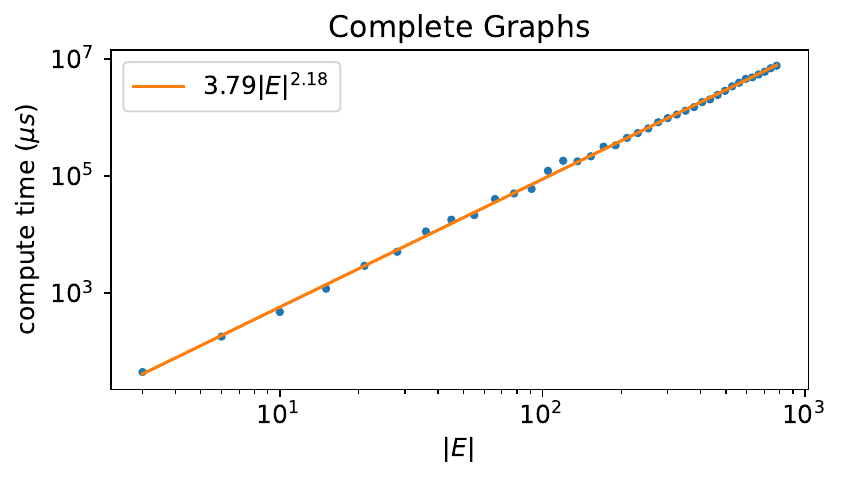}%
    \includegraphics[width=0.5\textwidth]{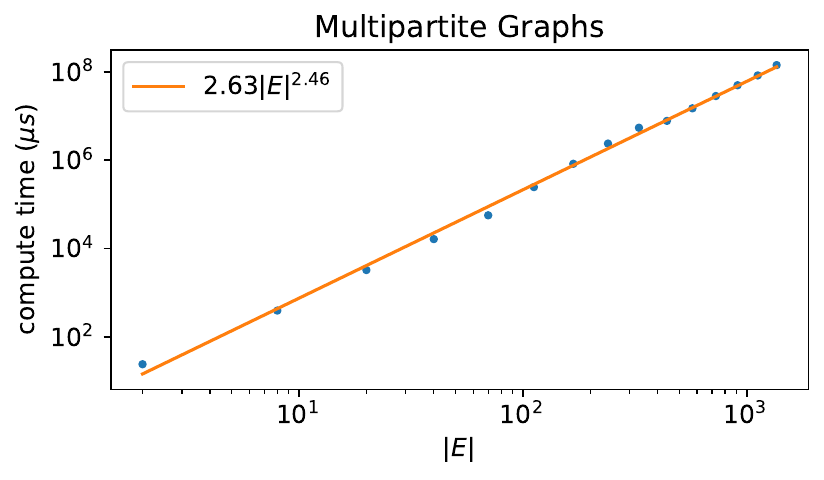}
    \includegraphics[width=0.5\textwidth]{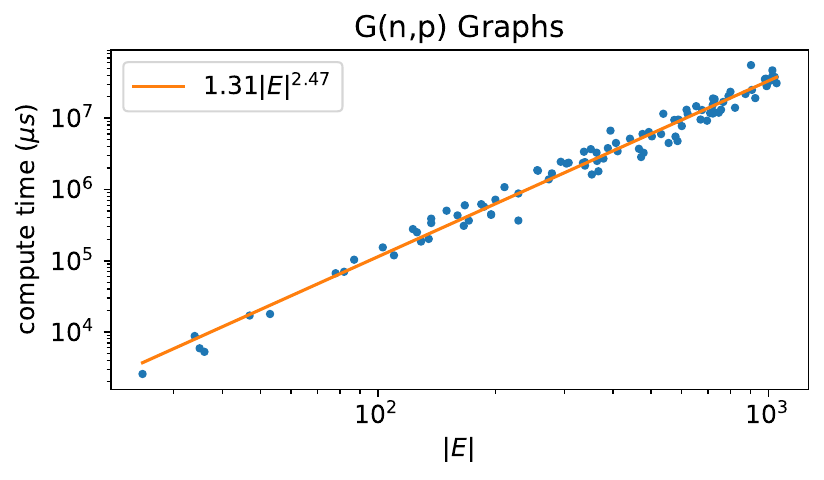}%
    \includegraphics[width=0.5\textwidth]{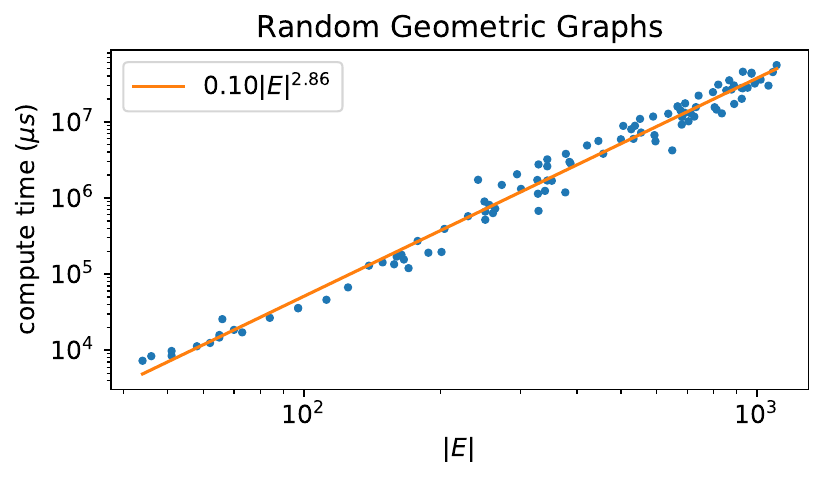}
\caption{Run-time complexity scaling tests as described in Section~\ref{sec:timing}.}
    \label{fig:time-complexity}
\end{figure}

\subsubsection{Modulus of the \emph{C.~elegans} metabolic network}\label{sec:celegans}

As a final example, we consider the spanning tree modulus of the graph formed from the \emph{C.~elegans} metabolic network, found in~\cite{duch2005community}.  The data for this graph were downloaded from~\cite{celegans}.  After the network is symmetrized and self-loops are removed, the resulting undirected graph has 453 vertices, 2025 edges, and approximately $6.6\times 10^{329}$ spanning trees.  The optimal $\eta^*$ on this graph takes 32 distinct values.  Figure~\ref{fig:celegans} provides a visualization of these values.  Each edge is colored based on its $\eta^*$ value.  Each vertex is sized and colored based on the smallest value of $\eta^*$ among its incident edges.  (Larger vertices correspond to smaller values of $\eta^*$.) The hierarchical structure of the graph can be seen in the figure; the highly connected ``core'' vertices are seen on the left side of the figure, while the more loosely connected peripheral vertices are seen on the right.

\begin{figure}
    \centering
    \includegraphics[width=\textwidth]{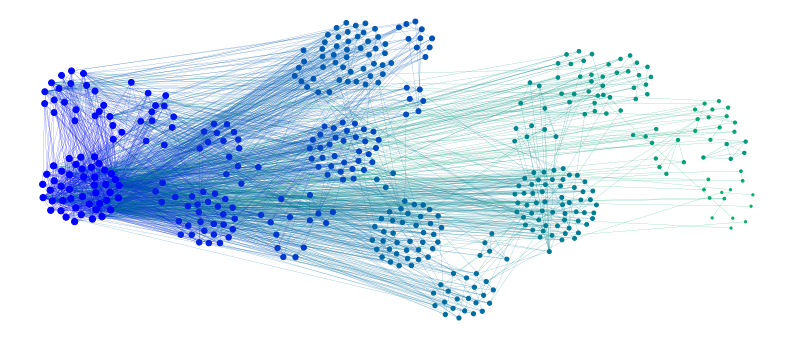}
    \caption{Visualization of the spanning tree modulus for the \emph{C.~elegans} metabolic network as described in Section~\ref{sec:celegans}.}
    \label{fig:celegans}
\end{figure}

\section*{Acknowledgments}
This material is based upon work supported by the National Science Foundation under grants No.~1515810 and No.~2154032.

\bibliographystyle{acm}
\bibliography{references}
\end{document}